\documentclass[11pt]{article}
\usepackage{amsmath,amsthm}
\usepackage{amssymb,mathrsfs}
\usepackage{bm,mathtools}

\usepackage{xcolor}
\usepackage{geometry}
\usepackage[colorlinks=true,
linkcolor=blue,citecolor=blue,
urlcolor=blue]{hyperref}

\makeatletter
\def\@seccntDot{.}
\def\@seccntformat#1{\csname the#1\endcsname\@seccntDot\hskip 0.5em}
\renewcommand\section{\@startsection{section}{1}{\z@}%
{18\p@ \@plus 6\p@ \@minus 3\p@}%
{9\p@ \@plus 6\p@ \@minus 3\p@}%
{\large\bfseries\boldmath}}
\renewcommand\subsection{\@startsection{subsection}{2}{\z@}%
{12\p@ \@plus 6\p@ \@minus 3\p@}%
{3\p@ \@plus 6\p@ \@minus 3\p@}%
{\bfseries\boldmath}}
\renewcommand\subsubsection{\@startsection{subsubsection}{3}{\z@}%
{12\p@ \@plus 6\p@ \@minus 3\p@}%
{\p@}%
{\bfseries\boldmath}}
\makeatother

\usepackage{microtype}

\theoremstyle{plain}
\newtheorem{theorem}{Theorem}[section]
\newtheorem{lemma}{Lemma}[section]
\newtheorem{problem}{Problem}[section]
\newtheorem{corollary}{Corollary}[section]

\newtheorem{conjecture}{Conjecture}[section]

\theoremstyle{definition}

\newtheorem{remark}{Remark}[section]

\numberwithin{equation}{section}
\allowdisplaybreaks
\DeclareMathOperator{\ex}{ex}
\DeclareMathOperator{\EX}{EX}
\DeclareMathOperator{\spex}{spex}
\DeclareMathOperator{\SPEX}{SPEX}

\title{Spectral Tur\'an-type problems on sparse spanning graphs}

\author{Lele Liu\footnote{School of Mathematical Sciences, Anhui University, Hefei 230601, 
P.R. China. E-mail: \texttt{liu@ahu.edu.cn} (L. Liu). Supported by the National
Nature Science Foundation of China (No. 12001370)}
~~ and ~~
Bo Ning\footnote{Corresponding author. College of Computer Science, Nankai University, Tianjin 300350, P.R. China.
E-mail: \texttt{bo.ning@nankai.edu.cn} (B. Ning). Partially supported by the National Nature Science
Foundation of China (No. 11971346).}}

\date{}

\begin{document}
\maketitle

\begin{abstract}
Let $F$ be a graph and $\SPEX (n, F)$ be the class of $n$-vertex graphs which
attain the maximum spectral radius and contain no $F$ as a subgraph. Let 
$\EX (n, F)$ be the family of $n$-vertex graphs which contain maximum number 
of edges and no $F$ as a subgraph. It is a fundamental problem in spectral extremal
graph theory to characterize all graphs $F$ such that $\SPEX (n, F)\subseteq \EX (n, F)$
when $n$ is sufficiently large. Establishing the conjecture of Cioab\u{a}, Desai 
and Tait [European J. Combin., 2022], Wang, Kang, and Xue [J. Combin. Theory Ser. B, 2023] 
prove that: for any graph $F$ such that the graphs in $\EX (n, F)$ are Tur\'{a}n 
graphs plus $O(1)$ edges, $\SPEX (n, F)\subseteq \EX (n, F)$ for sufficiently 
large $n$. In this paper, we prove that $\SPEX (n, F)\subseteq \EX (n, F)$ for 
sufficiently large $n$, where $F$ is an $n$-vertex graph with no isolated vertices 
and $\Delta (F) \leq \sqrt{n}/40$. We also prove a signless Laplacian spectral radius 
version of the above theorem. These results give new contribution to the open problem 
mentioned above, and can be seen as spectral analogs of a theorem of Alon and Yuster 
[J. Combin. Theory Ser. B, 2013]. Furthermore, as immediate corollaries, we have 
tight spectral conditions for the existence of several classes of special graphs, 
including clique-factors, $k$-th power of Hamilton cycles and $k$-factors in graphs. 
The first special class of graphs gives a positive answer to a problem of Feng,
and the second one extends a previous result of Yan et al.
\par\vspace{2mm}

\par\vspace{2mm}

\end{abstract}

\section{Introduction}
Extremal graph theory, one of the most important branches in graph theory, aims to 
characterize how global properties of a graph control the local structure of the graph. 
Given a graph $F$, let $\EX (n, F)$ be the family of $n$-vertex graphs with no copy 
of $F$ as a subgraph, containing the maximum number of edges. We denote by $\ex (n, F)$ 
the number of edges in a member of $\EX (n, F)$. One central problem in extremal graph 
theory is to study the behavior of the function $\ex (n, F)$ and to determine the 
classes of graphs in $\EX (n, F)$. One cornerstone result in this area is the Tur\'an 
Theorem \cite{Turan1941} in 1941, which states that the maximum number of edges in 
an $n$-vertex graph containing no $K_{r+1}$ as a subgraph equals to $\ex (T_{n,r})$, 
where $T_{n,r}$ is the $r$-partite Tur\'an graph, i.e., the complete $r$-partite 
graph such that every two parts have as equal size as possible. Erd\H{o}s, Stone and 
Simonovits \cite{ErdosStone1947,ErdosSimonovits1966} proved that
$\ex (n, F)=((1-1/r)/2+o(1)) n^2$ with given $\chi(F)=r+1\geq 2$ and sufficiently 
large $n$. From this result, one can see that $\ex (n,F)=o(n^2)$ for $\chi(F)=2$.
Till now, much of attention are paid on the study of Tur\'an functions of bipartite
graphs but little of exact results are obtained. For more development on extremal
graph theory, we refer the reader to \cite{Furedi-Simonovits2013}.

Compared with extremal graph theory, spectral extremal graph theory is a young but 
active branch of graph theory. Let $F$ be a given graph. We denote by $\SPEX_A(n,F)$ 
the class of graphs $G$ that attain the maximum adjacency spectral radius among all 
$n$-vertex graphs which do not contain $F$ as a subgraph. Let $\spex_A(n, F)$ be the 
spectral radius of graphs in $\SPEX_A(n, F)$ (when there is no danger of ambiguity, we
use $\SPEX (n, F)$ and $\spex (n, F)$ instead of $\SPEX_A(n, F)$ and $\spex_A (n, F)$, 
respectively). In this area, Nikiforov \cite{Nikiforov2011} proposed to study spectral 
analogous problems of Tur\'an-type problems, i.e., to study the maximum spectral radius 
among the class of $n$-vertex graphs containing no $F$ as a subgraph. In particular, 
Nikiforov \cite{Nikiforov2007} proved that $\SPEX (n, K_{r+1})\subseteq \EX (n, K_{r+1})$.
Although this result was reported by Guiduli \cite{G96} in his Ph.D. Thesis independently, 
Nikiforov later extended this result to the class of color critical graphs and published 
several papers including stronger spectral Tur\'an theorems in different directions, see
\cite{Nikiforov2009,Kang-Nikiforov2014}.

It is very natural to propose the following problem.

\begin{problem}\label{Prob:1}
Let $F$ be any graph. Characterize all graphs $F$ such that
\begin{equation}\label{eq:spex-subet-ex}
\SPEX (n, F)\subseteq \EX (n, F)
\end{equation}
for sufficiently large $n$.
\end{problem}

A plenty of work were published related to Problem \ref{Prob:1} which focuses on non-bipartite 
graphs. Given a graph $H$ and an integer $r \geq 3$, the \emph{edge blow-up} of $H$, denoted 
by $H^r$, is obtained by replacing each edge of $H$ with a clique of order $r$, where the new 
vertices of the cliques are all distinct. As a spectral analog of Tur\'an numbers of $S_{k+1}^3$ 
due to Erd\H{o}s et al. \cite{EFGG95}, Cioab\u{a}, Feng, Tait and Zhang \cite{CFTZ20} proved
\eqref{eq:spex-subet-ex} in Problem \ref{Prob:1} holds for $S_{k+1}^3$, where $S_{k+1}$ is the 
star with $k+1$ vertices, and the final spectral extremal graph was determined by Zhai, Liu and 
Xue \cite{ZLX22}. Let $L_{s,k}$ be the graph constructed by $s$ triangles and $k$ odd cycles of 
length at least $5$ which share one common vertex. Li and Peng \cite{LP22} and Desai et al. 
\cite{DesaiKang2022} extended the result of Cioab\u{a} et al. \cite{CFTZ20} to the classes of 
graphs $L_{s,k}$ and $S_k^r$, respectively. Lin, Zhai and Zhao \cite{Lin-Zhai-Zhao2022} proved 
that $\SPEX (n, \Gamma_k) \subseteq \EX (n, \Gamma_k)$ for $n$ large enough, where $\Gamma_k$ 
is the family of graphs without $k$-edge-disjoint triangles. Ni, Wang and Kang \cite{Ni-Wang-Kang2023} 
proved that $\SPEX (n, M_k^r) \subseteq \EX (n, M_k^r)$ for $n$ large enough, where $M_k$ is 
a matching of size $k$. Generalizing this result, Wang, Ni, Kang and Fan \cite{Wang-Ni-Kang-Fan2023} 
showed that \eqref{eq:spex-subet-ex} in Problem \ref{Prob:1} holds for the edge blow-up of star 
forest. On the \emph{wheel graph} $W_t$ of order $t$, i.e., the graph formed by joining a vertex 
to all of the vertices in a cycle on $(t-1)$ vertices, Cioab\u{a}, Desai and Tait \cite{CDT22} 
showed $\SPEX (n, W_5) \subseteq \EX (n, W_5)$. In the same paper, Cioab\u{a}, Desai and 
Tait \cite{CDT22} proposed a conjecture related to a more general phenomenon as follows.

\begin{conjecture}[Cioab\u{a}, Desai and Tait \cite{CDT22}]\label{Conj:1}
Let $F$ be any graph such that the graphs in $\EX (n, F)$ are Tur\'an graphs plus $O(1)$ edges. 
Then $\SPEX (n, F) \subseteq \EX (n, F)$ for $n$ large enough.
\end{conjecture}

Conjecture \ref{Conj:1} was confirmed by Wang, Kang and Xue \cite{WKX23} in a stronger form.

\begin{theorem}[Wang, Kang, and Xue \cite{WKX23}]\label{Thm:WKX23}
Let $r\geq 2$ be an integer, and $F$ be a graph with $\ex (n, F) = e(T_{n,r}) + O(1)$.
For sufficiently large $n$, we have $\SPEX (n, F)\subseteq \EX (n, F)$.
\end{theorem}

Theorem \ref{Thm:WKX23} gives us more new lights on spectral extremal graph theory, that is, 
we can study spectral extremal problems only with information on Tur\'an numbers of the graph $F$.

Compared with small graphs, the positive evidences for Problem \ref{Prob:1} also include 
sparse spanning subgraphs and large cycles. In this direction, with the help of results
of Ore \cite{Ore1961} and Bondy \cite{Bondy1972} respectively, Fiedler and Nikiforov 
\cite{Fielder-Nikiforov2010} proved that $\SPEX (n,C_n)\subseteq \EX (n,C_n)=\{K_1\vee (K_1\cup K_{n-2})\}$ 
for $n\geq 5$. Let $\Gamma$ be the collection of $2$-connected claw-free non-Hamiltonian 
$n$-vertex graphs with minimum degree at least $k$, let $\EX_{2\text{-con}}(n,C_n; \delta\geq k)$ 
be the class of graphs which attain the maximum number of edges among $\Gamma$, and 
$\SPEX_{2\text{-con}}(n,C_n; \delta\geq k)$ be the class of graphs which attain the 
maximum spectral radius among $\Gamma$. Li and Ning \cite{LN16} proved that 
$\SPEX (n,C_n;\delta\geq k)\subseteq \EX (n,C_n;\delta\geq k)$ when $n=\Omega(k^2)$.
For claw-free graphs, Li, Ning and Peng \cite{LNP18} proved that 
$\SPEX_{2\text{-con}}(n,\{C_n,K_{1,3}\};\delta\geq k)\subseteq \EX_{2\text{-con}}(n,\{C_n,K_{1,3}\};\delta\geq k)$.
Ge and Ning \cite{GN20} proved that \eqref{eq:spex-subet-ex} holds for $C_{n-1}$, which 
was improved by Li and Ning \cite{LN23} to that \eqref{eq:spex-subet-ex} holds for 
$C_{\ell}$, where $\ell$ is any integer in $[n-c_1\sqrt{n},n]$. For more results on 
large cycles supporting Problem \ref{Prob:1}, we refer the reader to \cite{LN23}.

In this paper, motivated by the phenomenon on Hamiltonicity of graphs, we contribute 
to Problem \ref{Prob:1} by proving a positive result when $F$ is a sparse spanning graph.
Let $H_{n,k}$ be an $n$-vertex graph consisting of an $(n-1)$-clique together with an 
additional vertex that is connected only to $(k - 1)$ vertices of the clique, that is, 
$H_{n,k} = K_{k - 1} \vee (K_{n-k} \cup K_1)$. One of our main results is as follows.

\begin{theorem}\label{thm:spectral-version-Alon-result}
Let $F$ be any $n$-vertex graph with no isolated vertices, $\delta (F) = \delta$ 
and $\Delta (F) \leq \sqrt{n}/40$. For all sufficiently large $n$, if $G$ is an 
$n$-vertex $F$-free graph, then $\lambda (G)\leq\lambda (H_{n, \delta})$
with equality holds if and only if $G = H_{n, \delta}$.
\end{theorem}

Our result can also be seen as a spectral analog of a theorem of Alon-Yuster \cite{Alon-Yuster2013}, 
whose proof is completely different from the extremal one.

\begin{theorem}[Alon-Yuster \cite{Alon-Yuster2013}]\label{thm:Alon-bound}
For all $n$ sufficiently large, if $F$ is any graph of order $n$ with no isolated vertices
and $\Delta (F) \leq \sqrt{n}/40$, then $ex(n,F) = \binom{n-1}{2}+\delta(F)-1$.
\end{theorem}

An immediate corollary of Theorem \ref{thm:spectral-version-Alon-result} and Theorem \ref{thm:Alon-bound} 
directly contributes to Problem \ref{Prob:1} positively.

\begin{corollary}
Let $F$ be any graph of order $n$ with no isolated vertices and $\Delta (F) \leq \sqrt{n}/40$.
Then $\SPEX (n, F)\subseteq \EX (n, F)$ holds for all sufficiently large $n$.
\end{corollary}

We also prove a $Q$-version of Theorem \ref{thm:spectral-version-Alon-result}, which requires 
a more involved proof including the use of the double eigenvectors technique. It is worth 
noting that this powerful technique can be traced back at least to Rowlinson \cite{Rowlinson1988}, 
and has been further developed in subsequent works, such 
as \cite{Chen-Wang-Zhai2022,Fan-Lin-Lu2022,Lou-Zhai2021,Zhai-Lin2022,Zhai-Xue-Lou2020}. 

\begin{theorem}\label{thm:q-spectral-version-Alon-result}
Let $F$ be any $n$-vertex graph with no isolated vertices, $\delta (F) = \delta$ and 
$\Delta (F) \leq \sqrt{n}/40$. For all sufficiently large $n$, if $G$ is an $n$-vertex 
$F$-free graph then $q (G)\leq q(H_{n, \delta})$, with equality holds if and only if 
$G = H_{n, \delta}$.
\end{theorem}

\section{Preliminaries}

In this section we introduce definitions and notation that will be used throughout the 
paper, and record several preparatory lemmas.

\subsection{Definitions and Notation}

Given a graph $G$ of order $n$, the \emph{adjacency matrix} $A(G)$ of $G$ is an $n$-by-$n$ 
matrix whose rows and columns are indexed by the vertices in $V(G)$. The $(i,j)$-entry of 
$A(G)$ is equal to $1$ if the vertices $i$ and $j$ are adjacent, and $0$ otherwise. Therefore, 
$A(G)$ is a real and symmetric matrix, it has $n$ real eigenvalues which we will denote by
$\lambda_1(G)\geq \lambda_2(G)\geq\cdots \geq \lambda_n(G)$. Let us recall that the signless 
Laplacian matrix of $G$ is defined as $Q(G) := D(G) + A(G)$, where $D(G)$ is the diagonal 
matrix whose entries are the degrees of the vertices of $G$. We shall write 
$q_1(G)\geq q_2(G)\geq\cdots\geq q_n(G)$ for the eigenvalues of $Q(G)$. We also write 
$\lambda(G):=\lambda_1(G)$ and $q(G) := q_1(G)$ for short. The Perron--Frobenius theorem 
for nonnegative matrices implies that $A(G)$ (resp. $Q(G)$) of a connected graph $G$ has 
a unique positive eigenvector of unit length corresponding to $\lambda(G)$ (resp. $q(G)$), 
and this eigenvector is called the \emph{Perron vector} of $A(G)$ (resp. $Q(G)$). 

Given a subset $X$ of the vertex set $V(G)$ of a graph $G$, we will let $G[X]$ be the 
subgraph of $G$ induced by $X$, and denote by $e(X)$ the number of edges in $G[X]$. 
We also use $e(G)$ to denote the number of edges of $G$. As usual, for a vertex $v$ 
of $G$ we write $d_G(v)$ and $N_G(v)$ for the degree of $v$ and the set of neighbors 
of $v$ in $G$, respectively. If the underlying graph $G$ is clear from the context, 
simply $d(v)$ and $N(v)$. We use the notations $\delta(G)$ and $\Delta(G)$ to represent, 
respectively, the minimum degree and maximum degree of $G$.

For a graph $G$, we denote the clique number of $G$ as $\omega(G)$, which represents 
the number of vertices in the largest complete subgraph of $G$. When considering two 
vertex-disjoint graphs, $G$ and $H$, we use $G \vee H$ to denote their \emph{join}, 
which is obtained by adding all possible edges between $G$ and $H$. The \emph{$k$-th power} 
of a graph $G$, denoted by $G^k$, is a graph with vertex set $V (G)$ in which two vertices 
are adjacent if and only if their distance is at most $k$ in $G$. For graph notation 
and terminology undefined here, we refer the reader to \cite{Bondy-Murty2008}.

\subsection{Basic lemmas}

We will use the following upper bound on $\lambda (G)$, which was proved by Hong, Shu,
and Fang \cite{HongShuFang2001} for connected graphs. Nikiforov \cite{Nikiforov2002} 
proved it for general graphs independently.

\begin{lemma}[\cite{HongShuFang2001,Nikiforov2002}]\label{lem:Hong-Nikiforov-bound}
Let $G$ be an $n$-vertex graph with $m$ edges. Then
\[
\lambda(G) \leq \frac{\delta(G) - 1}{2} + \sqrt{2m - \delta(G) n + \frac{(\delta(G) + 1)^2}{4}}.
\]
\end{lemma}

Let $G$ be a graph on $n$ vertices and $m$ edges. For any vector $(z_1,z_2,\ldots,z_n)$ 
with $z_i\geq 0$ and $\sum_{i=1}^n z_i = 1$, the well-known Motzkin--Straus 
inequality \cite{Motzkin-Straus1965} states that
\begin{equation}\label{eq:Motzkin-Straus-inequality}
2 \sum_{ij\in E(G)} z_iz_j \leq 1 - \frac{1}{\omega (G)}.
\end{equation}
Now, let $\bm{x} = [x_i]$ be a nonnegative vector of unit length, and set
\[
\bm{y} := \frac{\bm{x}}{\|\bm{x}\|_1}.
\]
Obviously, $\|\bm{y}\|_1 = 1$. In light of \eqref{eq:Motzkin-Straus-inequality}, we find
\begin{equation}\label{eq:clique-vector}
1 - \frac{1}{\omega(G)} \geq 2\sum_{ij\in E(G)} y_iy_j 
= \frac{2}{\|\bm{x}\|_1^2} \cdot \sum_{ij\in E(G)} x_ix_j.
\end{equation}
Hence, we immediately obtain the following result, established by Wilf.

\begin{lemma}[\cite{Wilf1986}]\label{lem:Wilf-bound}
Let $\bm{x}$ be the Perron vector of $A(G)$. Then
\[
\lambda (G) \leq \|\bm{x}\|_1^2 \cdot \Big(1 - \frac{1}{\omega (G)}\Big).
\]
\end{lemma}

The next lemma gives us a pithy bound on the largest eigenvalue of $Q(G)$.

\begin{lemma}[\cite{FengYu2009}]\label{lem:q1-upper-bound}
Let $G$ be a graph with $n$ vertices and $m$ edges. Then 
\[
q(G) \leq \frac{2m}{n-1} + n-2.
\]
\end{lemma}

Finally, we also need the following double eigenvectors technique for signless 
Laplacian matrices of graphs. 

\begin{lemma}[\cite{Zhai-Xue-Lou2020}]\label{lem:double-eigenvectors}
Let $G$ and $H$ be two graphs with $|V(G)| = |V(H)|$. Let $\bm{x}$ and $\bm{y}$ be 
the Perron vectors of $Q(G)$ and $Q(H)$, respectively. Then
\[
\bm{x}^{\mathrm{T}} Q(G) \bm{y} = \sum_{ij\in E(G)} (x_i + x_j) (y_i + y_j),
\]
and
\[
\bm{x}^{\mathrm{T}} \bm{y} (q(H) - q(G)) = \bm{x}^{\mathrm{T}} (Q(H) - Q(G)) \bm{y}.
\]
\end{lemma}

\section{Proof of Theorem \ref{thm:spectral-version-Alon-result}}

The aim of this section is to give a proof of Theorem \ref{thm:spectral-version-Alon-result}.
Assume that $F$ is an $n$-vertex graph with no isolated vertices and $\Delta (F) \leq \sqrt{n}/40$.
Let $G$ be a graph with maximum spectral radius among all $n$-vertex graphs which contain no 
copy of $F$ as a spanning subgraph, and $\bm{x}$ be the Perron vector of $A(G)$. With this 
notation, for any $v \in V(G)$, the eigenvalue equation with respect to $v$ becomes
\[
\lambda (G) x_v = \sum_{u\in N(v)} x_u.
\]
Throughout this section, we set $m:= |E(G)|$, $\delta (F):= \delta$ and 
$x_{\max}:= \max\{x_u: u\in V(G)\}$ for short.

We commence with a simple lemma that, while not optimal, adequately fulfills our requirements.

\begin{lemma}\label{lem:lower-bound-lambda}
$\lambda (G) \geq n - 2$.
\end{lemma}

\begin{proof}
Since $H_{n, \delta}$ contains no $F$ as a spanning subgraph, we obtain
\[
\lambda (G) \geq \lambda(H_{n, \delta}) \geq \lambda (K_{n-1}) = n - 2,
\]
as desired.
\end{proof}

With the help of Lemma \ref{lem:Hong-Nikiforov-bound} and Lemma \ref{lem:lower-bound-lambda}, 
we can derive a reasonable lower bound on the size of $G$.

\begin{lemma}\label{lem:lower-bound-edges}
$m \geq \binom{n-1}{2} + \frac{\delta(G)}{2}$.
\end{lemma}

\begin{proof}
In view of Lemma \ref{lem:Hong-Nikiforov-bound}, we see
\[
\lambda (G) \leq \frac{\delta(G) - 1}{2} + \sqrt{2m - \delta(G) n + \frac{(\delta(G) + 1)^2}{4}}.
\]
On the other hand, $\lambda (G) \geq n - 2$ by Lemma \ref{lem:lower-bound-lambda}. Hence,
\[
n - 2 \leq \frac{\delta(G) - 1}{2} + \sqrt{2m - \delta(G) n + \frac{(\delta(G) + 1)^2}{4}},
\]
Solving the above inequality, we obtain the desired result.
\end{proof}

\begin{lemma}\label{lem:upper-lower-bound-minimum-degree}
$\delta (F) - 1 \leq \delta (G) \leq 2(\delta (F) - 1)$.
\end{lemma}

\begin{proof}
We first prove the left-hand side. Assume that $u$ is a vertex of $G$ such that $d(u) = \delta(G)$, 
if $\delta(G) < \delta (F) - 1$, we can add an edge which joins $u$ and a vertex in $V(G)\setminus N(u)$ 
to $G$. The resulting graph has larger spectral radius and still contains no $F$ as a subgraph. 
This is a contradiction.

For the right-hand side, by Lemma \ref{lem:lower-bound-edges}, $m\geq\binom{n-1}{2} + \delta (G)/2$.
On the other hand, $m \leq \binom{n-1}{2} + \delta (F) - 1$ by Theorem \ref{thm:Alon-bound}. Thus,
\[
\binom{n-1}{2} + \frac{\delta (G)}{2} \leq m \leq \binom{n-1}{2} + \delta (F) - 1,
\]
completing the proof of Lemma \ref{lem:upper-lower-bound-minimum-degree}.
\end{proof}

Based on Lemma \ref{lem:lower-bound-edges}, it can be deduced that there is at most one vertex 
with degree $o(n)$. Moreover, combining this result with Lemma \ref{lem:upper-lower-bound-minimum-degree}, 
we conclude that there exists precisely one vertex with degree $o(n)$. For the subsequent 
discussion, we assume that $w$ is the unique vertex such that $d(w) = \delta (G)$. As a 
result, we have $d(w) \leq 2(\delta (F)-1)$.

\begin{lemma}\label{lem:large-degree}
For each $v\in V(G)\setminus\{w\}$, we have $d(v) \geq n - 2 - \delta (G)$.
\end{lemma}

\begin{proof}
Assume by contradiction that there is a vertex $v_0\in V(G)\setminus\{w\}$ such that 
$d(v_0) < n - 2 - \delta (G)$. Then
\begin{align*}
\sum_{u\in V(G)} d(u)
& \leq d(v_0) + d(w) + (n - 2 - d(w)) (n - 2) + d(w) (n - 1) \\
& = d(v_0) + 2d(w) + (n - 2)^2 \\
& = d(v_0) + 2\delta (G) + (n - 2)^2 \\
& < (n - 1)(n - 2) + \delta (G).
\end{align*}
On the other hand, by Lemma \ref{lem:lower-bound-edges} we have
\[
\sum_{u\in V(G)} d(u) = 2m \geq (n - 1)(n - 2) + \delta (G),
\]
a contradiction. This completes the proof.
\end{proof}

Now, we shall present several lemmas concerning the Perron vector $\bm{x}$ of $A(G)$. 
The next lemma, roughly speaking, demonstrates that most vertices of $G$ have eigenvector 
entries approximately $n^{-1/2}$.

\begin{lemma}\label{lem:xmax-upper-bound}
$x_{\max} \leq \frac{\sqrt{n}}{n - 1}$.
\end{lemma}

\begin{proof}
Assume that $u$ is a vertex such that $x_u = x_{\max}$. It follows from $\|\bm{x}\|_2=1$ and
Cauchy--Schwarz inequality that
\[
(1 + \lambda (G)) x_u = x_u + \sum_{uv\in E(G)} x_v \leq \sum_{v\in V(G)} x_v \leq \sqrt{n}.
\]
On the other hand, $\lambda (G) \geq n - 2$ by Lemma \ref{lem:lower-bound-lambda}, implying 
$(n-1) x_u \leq \sqrt{n}$. This completes the proof of Lemma \ref{lem:xmax-upper-bound}.
\end{proof}

\begin{lemma}\label{lem:sum-eigenvector}
$\|\bm{x}\|_1 \geq \sqrt{n - 1}$.
\end{lemma}

\begin{proof}
By Lemma \ref{lem:Wilf-bound} and Lemma \ref{lem:lower-bound-lambda}, we deduce that
\[
n - 2 \leq \lambda (G) \leq \|\bm{x}\|_1^2 \cdot \Big(1 - \frac{1}{\omega(G)}\Big).
\]
Since $\omega (G) \leq n - 1$, we see
\[
n - 2 \leq \|\bm{x}\|_1^2 \cdot \Big(1 - \frac{1}{n - 1}\Big).
\]
Solving this inequality, we obtain the desired result.
\end{proof}

With the support of Lemma \ref{lem:xmax-upper-bound} and Lemma \ref{lem:sum-eigenvector}, we 
can show that all vertices of $G$, except for $w$, have large eigenvector entries.

\begin{lemma}\label{lem:lower-bound-large-component}
For each $v\in V(G)\setminus \{w\}$, we have
\[
x_v > \frac{9}{10\sqrt{n}}.
\]
\end{lemma}

\begin{proof}
In light of the eigenvalue equation with respect to $v$, we have
\[
\lambda (G) x_v = \sum_{u\in N(v)} x_u = \sum_{u\in V(G)} x_u - \sum_{u\notin N(v)} x_u.
\]
By Lemma \ref{lem:xmax-upper-bound} and Lemma \ref{lem:sum-eigenvector}, we deduce that
\[
\lambda (G) x_v \geq \sqrt{n - 1} - \frac{\sqrt{n}}{n-1}\cdot (n - d(v)).
\]
Since $d(v) \geq n - 2 - \delta (G)$ by Lemma \ref{lem:large-degree},
\[
\lambda (G) x_v > \sqrt{n - 1} - \frac{\sqrt{n}}{n-1}\cdot (\delta (G) + 2).
\]
On the other hand, noting that $\lambda (G) < n - 1$ and $\delta(G) +2 \leq 2\delta(F) \leq \sqrt{n}/20$ 
by Lemma \ref{lem:upper-lower-bound-minimum-degree}, we have
\[
x_v > \frac{1}{\sqrt{n-1}} - \frac{\sqrt{n}}{(n-1)^2}\cdot (\delta (G) + 2) 
> \frac{9}{10 \sqrt{n}},
\]
completing the proof of Lemma \ref{lem:lower-bound-large-component}. 
\end{proof}

We now prove that the remaining vertex $w$ has small eigenvector entry.

\begin{lemma}\label{lem:upper-bound-small-component}
$x_w < \frac{1}{19n}$.
\end{lemma}

\begin{proof}
By eigenvalue equation for $w$, we have
\[
\lambda (G) x_w = \sum_{vw\in E(G)} x_v \leq \delta (G) \cdot x_{\max}.
\]
On the other hand, $\lambda (G)\geq n-2$ by Lemma \ref{lem:lower-bound-lambda} and
$x_{\max} \leq \frac{\sqrt{n}}{n - 1}$ by Lemma \ref{lem:xmax-upper-bound}. Therefore,
\[
(n-2) x_w \leq \frac{\sqrt{n}}{n - 1} \cdot \delta (G)
\leq \frac{\sqrt{n}}{n - 1} \cdot \frac{\sqrt{n}}{20}.
\]
The result follows by solving the above inequality.
\end{proof}

The final lemma almost determines the structure of the extremal graph.

\begin{lemma}\label{lem:induced-subgraph}
The induced subgraph $G[V(G)\setminus \{w\}]$ is the complete graph $K_{n-1}$.
\end{lemma}

\begin{proof}
By way of contradiction, assume that there exist $u,v\in V(G)\setminus \{w\}$ such that 
$uv\notin E(G)$. Let $G'$ be a graph obtained from $G$ by removing $\delta (G) - \delta (F) + 1$ 
edges incident to $w$ (denote these edges by $E'$), and adding $uv$. Clearly, $G'$ contains 
no $F$ as a subgraph. By Rayleigh's principle, Lemma \ref{lem:xmax-upper-bound}, 
Lemma \ref{lem:lower-bound-large-component} and Lemma \ref{lem:upper-bound-small-component}, 
we deduce that
\begin{align*}
\lambda (G') - \lambda (G)
& \geq \bm{x}^{\mathrm{T}} A(G') \bm{x} - \bm{x}^{\mathrm{T}} A(G) \bm{x} \\
& = 2 x_u x_v - 2 \sum_{wt\in E'} x_w x_t\\
& > 2 \left(\frac{9}{10 \sqrt{n}}\right)^2 - 2\left(\frac{1}{19n}\cdot 
\frac{\sqrt{n}}{n-1}\right)\cdot (\delta(G)-\delta(F)+1) \\
& \geq 2 \left(\frac{9}{10 \sqrt{n}}\right)^2 - 2\left(\frac{1}{19n}\cdot 
\frac{\sqrt{n}}{n-1}\right) \cdot \frac{\sqrt{n}}{20} \\
& > 0,
\end{align*}
a contradiction completing the proof.
\end{proof}

We now combine the results from the previous lemmas to prove Theorem \ref{thm:spectral-version-Alon-result}. \par\vspace{2mm}

\noindent \emph{Proof of Theorem \ref{thm:spectral-version-Alon-result}.}
By Lemma \ref{lem:induced-subgraph}, it suffices to show $d(w) = \delta (F) - 1$.
If $d(w) > \delta (F) - 1$, then $G$ contains $F$ as a spanning subgraph by Lemma \ref{lem:induced-subgraph}.
So we see $d(w) \leq \delta (F) - 1$. On the other hand, $d(w) \geq \delta (F) - 1$
by Lemma \ref{lem:upper-lower-bound-minimum-degree}. This completes the proof of 
Theorem \ref{thm:spectral-version-Alon-result}.

\section{Proof of Theorem \ref{thm:q-spectral-version-Alon-result}}

In this section, we will prove Theorem \ref{thm:q-spectral-version-Alon-result}. Before 
giving the details, we introduce some notation to be used throughout the proof. In this 
section, we always assume that $G\in\SPEX (n, F)$ and $\bm{x}$ is the Perron vector of $Q(G)$.
For convenience, set $m:= |E(G)|$, $\delta:= \delta(F)$ and $x_{\max}:= \max\{x_u: u\in V(G)\}$.

We begin with a simple fact on the largest eigenvalue of $Q(G)$.

\begin{lemma}\label{lem:q-lower-bound-lambda}
$q(G) \geq 2(n-2) + \frac{\delta(F)-1}{n-1}$.
\end{lemma}

\begin{proof}
Define a vector $\bm{y}$ for $H_{n, \delta}$ as follows:
\[
y_u = 
\begin{cases}
\frac{\delta(F) - 1}{2n}, & d_{H_{n, \delta}} (u) = \delta(F) - 1, \\
1, & \text{otherwise}.
\end{cases}
\]
The Rayleigh's principle implies that
\begin{align*}
q(G) & \geq \frac{\bm{y}^{\mathrm{T}} Q(H_{n, \delta}) \bm{y}}{\|\bm{y}\|^2} 
= \frac{ 2(n-1)(n-2) + (\delta(F)-1)\cdot \big(1+\frac{\delta(F)-1}{2n} \big)^2 }{n-1 + \big(\frac{\delta(F)-1}{2n}\big)^2} \\
& \geq 2n-4 + \frac{\delta(F)-1}{n-1}.
\end{align*}
This completes the proof of Lemma \ref{lem:q-lower-bound-lambda}.
\end{proof}

\begin{lemma}\label{lem:q-lower-bound-edges}
$m \geq \binom{n-1}{2} + \frac{\delta(F)-1}{2}$.
\end{lemma}

\begin{proof}
In light of Lemma \ref{lem:q1-upper-bound}, we get
\[
q (G) \leq \frac{2m}{n-1} + n - 2.
\]
On the other hand, combining with Lemma \ref{lem:q-lower-bound-lambda} gives
\[
2(n - 2) + \frac{\delta(F)-1}{n-1} \leq \frac{2m}{n-1} + n - 2.
\]
Solving the above inequality, we obtain the desired result.
\end{proof}

\begin{lemma}\label{lem:q-upper-lower-bound-minimum-degree}
$\delta (G) \geq \delta (F) - 1$.
\end{lemma}

\begin{proof}
Assume to the contrary that $\delta(G) < \delta (F) - 1$. Let $u$ be a vertex such 
that $d(u) = \delta(G)$. Now, we add an edge which joining $u$ and a vertex in 
$V(G)\setminus N(u)$ to $G$. The resulting graph has larger signless Laplacian spectral
radius and still contains no $H$ as a subgraph. This is a contradiction.
\end{proof}

The next three lemmas focus on the eigenvector entries of the Perron vector $\bm{x}$ of $Q(G)$.

\begin{lemma}\label{lem:q-xmax-upper-bound}
$x_{\max} \leq \frac{\sqrt{n}}{n - 2}$.
\end{lemma}

\begin{proof}
Assume that $u$ is a vertex such that $x_u = x_{\max}$. Using the eigenvalue equation 
with respect to the vertex $u$, we see
\[
(q (G) - d(u)) x_u = \sum_{uv\in E(G)} x_v.
\]
It follows from $\|\bm{x}\|_2=1$ and Cauchy--Schwarz inequality that
\[
(q (G) - d(u) + 1) x_u = x_u + \sum_{uv\in E(G)} x_v \leq \sum_{v\in V(G)} x_v \leq \sqrt{n}.
\]
On the other hand, $q (G) \geq 2(n - 2)$ by Lemma \ref{lem:q-lower-bound-lambda}, 
implying $(n-2) x_u \leq \sqrt{n}$.
\end{proof}

\begin{lemma}\label{lem:q-sum-eigenvector}
$\|\bm{x}\|_1 \geq \sqrt{n - 2}$.
\end{lemma}

\begin{proof}
Since $\bm{x}$ is the Perron vector of $Q(G)$, we have
\[
q(G) = \bm{x}^{\mathrm{T}} Q(G) \bm{x} = \sum_{v\in V(G)} d(v) x_v^2 + 2\sum_{uv\in E(G)} x_ux_v.
\]
Noting that $\|\bm{x}\|_2^2 = 1$, we find that 
\begin{align*}
q(G) & \leq \Delta(G) \cdot \sum_{v\in V(G)} x_v^2 + 2\sum_{uv\in E(G)} x_ux_v \\
& \leq n - 1 + 2\sum_{uv\in E(G)} x_ux_v.
\end{align*}
Combining this with \eqref{eq:clique-vector} gives
\begin{align*}
q (G) 
& \leq n - 1 + \|\bm{x}\|_1^2 \cdot \Big(1 - \frac{1}{\omega(G)}\Big) \\
& \leq n - 1 + \frac{n-3}{n-2} \cdot \|\bm{x}\|_1^2,
\end{align*}
where the last inequality using the fact 
that $\omega (G) \leq n - 2$. On the other hand, $q(G) \geq 2(n-2)$ by Lemma \ref{lem:q-lower-bound-lambda}, 
we conclude that
\[
2(n-2) \leq n-1 + \frac{n-3}{n-2} \cdot \|\bm{x}\|_1^2.
\]
Solving this inequality we obtain the desired result.
\end{proof}

\begin{lemma}\label{lem:relation-xv-dv}
For each vertex $v$, let $c_v:= d(v)/n$. Then
\[
x_v = \frac{c_v}{(2-c_v)\sqrt{n}} + o\Big(\frac{1}{\sqrt{n}}\Big).
\]
\end{lemma}

\begin{proof}
By the eigenvalue equation with respect to the vertex $v$ and Lemma \ref{lem:q-xmax-upper-bound} 
we obtain 
\[
(q(G) - d(v)) x_v = \sum_{u\in N(v)} x_u < \frac{\sqrt{n}}{n-2} \cdot d(v).
\]
Combining with Lemma \ref{lem:q-lower-bound-lambda} we get
\[
x_v \leq \frac{\sqrt{n} \cdot d(v)}{(n-2) (2n - 4 - d(v))} 
= \frac{c_v}{(2-c_v)\sqrt{n}} + o\Big(\frac{1}{\sqrt{n}}\Big).
\]

On the other hand, using the eigenvalue equation for $v$ again gives
\[
(q (G) - d(v)) x_v = \sum_{u\in N(v)} x_u = \|\bm{x}\|_1 - \sum_{u\notin N(v)} x_u.
\]
By Lemma \ref{lem:q-xmax-upper-bound} and Lemma \ref{lem:q-sum-eigenvector}, we deduce that
\begin{align*}
(q (G) - d(v)) x_v 
& \geq \sqrt{n - 2} - \frac{\sqrt{n}}{n-2}\cdot (n - d(v)) \\
& > \frac{d(v)}{\sqrt{n}} - \frac{3\sqrt{n}}{n-2}.
\end{align*}
Dividing both sides by $n$ and using $q(G) < 2n$, we find that
\[
(2 - c_v) x_v > \frac{c_v}{\sqrt{n}} - \frac{3}{\sqrt{n} (n-2)}.
\]
As a consequence,
\[
x_v > \frac{c_v}{(2-c_v)\sqrt{n}} - \frac{3}{\sqrt{n} (n-2)},
\]
completing the proof of Lemma \ref{lem:relation-xv-dv}.
\end{proof}

Fix a sufficiently small constant $0 < \varepsilon < 1/7$, we denote 
\[
L:= \{v\in V(G): d(v)> (1-\varepsilon)n\},~~S:= V(G)\setminus L.
\]
With the notation above we first show that the size of $S$ is small.

\begin{lemma}
$|S| < 3/\varepsilon$. 
\end{lemma}

\begin{proof}
By definition of $L$, we have
\begin{align*}
2e(G) & = \sum_{v\in L} d(v) + \sum_{v\in S} d(v) \\
& < n\cdot |L| + (1-\varepsilon) n\cdot |S| \\
& = n(n-|S|) + (1-\varepsilon) n\cdot |S| \\
& = n^2 - \varepsilon n\cdot |S|.
\end{align*}
On the other hand, it follows from Lemma \ref{lem:q-lower-bound-edges} that
\[
n^2 - \varepsilon n \cdot |S| > 2\binom{n-1}{2}.
\]
Solving the above inequality we find $|S| < 3/\varepsilon$, as desired.
\end{proof}

Let $w$ be a vertex such that $x_w = \min\{x_v : v\in V(G)\}$. Next lemma shows 
that the vertex degree of $w$ is small.

\begin{lemma}\label{lem:small-degree-w}
$d(w) < \delta(F) + 14/\varepsilon^2$.
\end{lemma}

\begin{proof}

We assume towards contradiction that $d(w) \geq \delta(F) + 14/\varepsilon^2$. 
Consider the $F$-free graph $H_{n, \delta}$. Obviously, $H_{n,\delta}$ can be 
obtained from $G$ by removing $d(w)-\delta(F)+1$ edges incident with $w$ 
(denote the set of edges by $E_1$), and adding $\binom{n-1}{2} + d(w) - e(G)$ 
pairs from $V(G)\setminus \{w\}$ (denote the set of these edges by $E_2$). 
Since $e(G) \leq \binom{n-1}{2} + \delta(F) - 1$, we have 
\begin{equation}\label{eq:E2-E1}
|E_2|\geq |E_1| > \frac{14}{\varepsilon^2}.
\end{equation}

In what follows, we shall prove that $q(H_{n, \delta}) > q(G)$ using the double 
eigenvectors technique for signless Laplacian matrices of graphs, and therefore 
get a contradiction. To this end, let $\bm{y}$ be the Perron vector of $H_{n, \delta}$. 
By some straightforward computation, we obtain
\begin{equation}\label{eq:eigenvector-y}
y_v = 
\begin{dcases}
(1+o(1)) \frac{1}{\sqrt{n}}, & v\neq w, \\
o\Big(\frac{1}{\sqrt{n}}\Big), & v = w.
\end{dcases}
\end{equation}
On the other hand, in view of Lemma \ref{lem:relation-xv-dv}, for each $v\in L$ we have
\begin{equation}\label{eq:xv-lower-bound-for-L}
x_v > \Big(\frac{1-\varepsilon}{1+\varepsilon} + o(1)\Big) \frac{1}{\sqrt{n}}.
\end{equation}
Using Lemma \ref{lem:double-eigenvectors} we get
\begin{align*}
\bm{x}^{\mathrm{T}} \bm{y} (q(H_{n, \delta}) - q(G)) 
& = \bm{x}^{\mathrm{T}} (Q(H_{n, \delta}) - Q(G)) \bm{y} \\
& = \sum_{ij\in E_2} (x_i + x_j) (y_i + y_j) - \sum_{ij\in E_1} (x_i + x_j) (y_i + y_j) \\
& \geq \sum_{ij\in E_2\setminus E(S)} (x_i + x_j) (y_i + y_j) - \sum_{ij\in E_1} (x_i + x_j) (y_i + y_j).
\end{align*}
To find the first term in the right side of the last inequality, note that \eqref{eq:E2-E1}, 
\eqref{eq:eigenvector-y} and \eqref{eq:xv-lower-bound-for-L},
\begin{align*}
\sum_{ij\in E_2\setminus E(S)} (x_i + x_j) (y_i + y_j) 
& > \Big(|E_1|-\frac{9}{2\varepsilon^2}\Big) \bigg(x_w + \Big(\frac{1-\varepsilon}{1+\varepsilon} + o(1)\Big)\frac{1}{\sqrt{n}}\bigg) \cdot(2+o(1))\frac{1}{\sqrt{n}} \\
& = \frac{1}{\sqrt{n}} \Big(|E_1|-\frac{9}{2\varepsilon^2}\Big) \bigg( 2x_w + \Big( \frac{2(1-\varepsilon)}{1+\varepsilon} + o(1)\Big) \frac{1}{\sqrt{n}}\bigg).
\end{align*}
Similarly, to find the second term in the right side of the last inequality, note 
that \eqref{eq:eigenvector-y} and Lemma \ref{lem:q-xmax-upper-bound},
\begin{align*}
\sum_{ij\in E_1} (x_i + x_j) (y_i + y_j)
& < |E_1| \Big( x_w + (1+o(1))\frac{1}{\sqrt{n}} \Big) \cdot (1+o(1)) \frac{1}{\sqrt{n}} \\
& = \frac{|E_1|}{\sqrt{n}} \bigg( x_w + (1+o(1)) \frac{1}{\sqrt{n}} \bigg).
\end{align*}
Combining these two inequalities, and noting that $|E_1| > 14/\varepsilon^2$ by \eqref{eq:E2-E1}, 
we obtain 
\begin{align*}
\bm{x}^{\mathrm{T}} \bm{y} (q(H_{n, \delta}) - q(G)) \cdot \sqrt{n}
> & ~|E_1| \bigg( x_w + \Big( \frac{1-3\varepsilon}{1+\varepsilon} + o(1) \Big) \frac{1}{\sqrt{n}}\bigg) \\
& ~- \frac{9}{2\varepsilon^2} \bigg( 2x_w + \Big( \frac{2(1-\varepsilon)}{1+\varepsilon} + o(1)\Big) \frac{1}{\sqrt{n}} \bigg) \\
> & ~\frac{9}{2\varepsilon^2} \bigg( x_w + \Big( \frac{1-7\varepsilon}{1+\varepsilon} + o(1)\Big) \frac{1}{\sqrt{n}} \bigg) \\
> & ~0,
\end{align*}
which yields that $q(H_{n, \delta}) > q(G)$, a contradiction completing the proof.
\end{proof}

\begin{lemma}\label{lem:induced-subgraph-Q}
The induced subgraph $G[V(G)\setminus \{w\}]$ is the complete graph $K_{n-1}$.
\end{lemma}

\begin{proof}
If $\delta (G) = \delta(F) - 1$, the conclusion is clear. By way of contradiction, 
assume that there exist $G[V(G)\setminus \{w\}]$ is not a complete graph. Let $G'$ 
be a graph obtained from $G$ by removing $\delta (G) - \delta (F) + 1$ edges incident 
to $w$ (denote these edges by $S_1$), and adding $\delta (G) - \delta (F) + 1$ pairs 
from $V(G)\setminus\{w\}$ (denote these edges by $S_2$). Clearly, $G'$ contains no 
$F$ as a subgraph. By Lemma \ref{lem:relation-xv-dv} and Lemma \ref{lem:small-degree-w}, 
we deduce that $x_w = o(n^{-1/2})$ and $x_v = (1+o(1)) n^{-1/2}$ for each $v\in V(G)\setminus\{w\}$. 
Hence,
\begin{align*}
q(G') - q(G)
& \geq \bm{x}^{\mathrm{T}} Q(G') \bm{x} - \bm{x}^{\mathrm{T}} Q(G) \bm{x} \\
& = \sum_{uv\in S_2} (x_u + x_v)^2 - \sum_{wu\in S_1} (x_w + x_u)^2 \\
& = (\delta (G) - \delta (F) + 1) \left(\left( \frac{2+o(1)}{\sqrt{n}} \right)^2 - \left(\frac{1+o(1)}{\sqrt{n}} \right)^2\right) \\
& > 0,
\end{align*}
a contradiction completing the proof.
\end{proof}

We are now ready to complete the proof of Theorem \ref{thm:q-spectral-version-Alon-result}. 
\par\vspace{2mm}

\noindent \emph{Proof of Theorem \ref{thm:q-spectral-version-Alon-result}.}
By Lemma \ref{lem:induced-subgraph-Q}, it suffices to show $d(w) = \delta (F) - 1$.
If $d(w) > \delta (F) - 1$, then $G$ contains $F$ as a spanning subgraph.
So we see $d(w) \leq \delta (F) - 1$. On the other hand, $d(w) \geq \delta (F) - 1$
by Lemma \ref{lem:q-upper-lower-bound-minimum-degree}.

\section{Concluding remarks}

In this paper, we prove $\SPEX (n,F)\subseteq \EX(n,F)$ when $F$ is a spanning graph 
without isolated vertices and $\Delta (F) \leq \sqrt{n}/40$, where $n$ is sufficiently 
large. The immediate corollaries include tight spectral condition for the $k$-th power 
of Hamilton cycle and $[a, b]$-factors.

\subsection{Corollaries}
Let $C^k_n$ be the $k$-th power of Hamilton cycle, i.e, the $n$-cycle.
Setting $F = C^k_n$ in Theorem \ref{thm:spectral-version-Alon-result}, we have

\begin{theorem}
Let $G$ be an $n$-vertex graph not containing $C^k_n$ as a subgraph. Then there exists 
an integer $n_0$ such that if $n\geq \max\{(80k)^2,n_0\}$, then $\lambda (G)\leq\lambda (H_{n, 2k})$,
with equality holds if and only if $G = H_{n, 2k}$.
\end{theorem}

This theorem extends a result due to Yan et al. (see \cite[Corollary 1.6]{YHFL23}) 
while also providing a solution to a question raised within the same paper. 

\begin{problem}[\cite{YHFL23}]
At last, we expect that the extremal graph without containing $C^k_n$
for $n$ large enough may be the graph $K_n\backslash E(S_{n-2k+1})$.
\end{problem}

Our results are also related to the existence of $[a, b]$-factors of graphs. 
An \emph{$[a, b]$-factor} of a graph $G$ is a spanning subgraph $H$ such that 
$a \leq d_H (v) \leq b$ for each $v \in V(G)$. Furthermore, if $a=b=k$, then 
$H$ is call a \emph{$k$-factor} of $G$. Using Theorem \ref{thm:spectral-version-Alon-result} 
we immediately have the following.

\begin{theorem}
Let $G$ be an $n$-vertex graph, and $a,b$ be integers such that $1\leq a\leq b \leq \sqrt{n}/40$. 
For sufficiently large $n$, we have
\begin{enumerate}
\item[$(1)$] if $\lambda(G) \geq \lambda (H_{n, a})$, then $G$ contains an $[a, b]$-factor 
unless $G\cong H_{n, a}$.

\item[$(2)$] if $\lambda(G) \geq \lambda (H_{n, b})$, then $G$ contains all $[1, b]$-factors 
unless $G\cong H_{n, b}$.
\end{enumerate}
\end{theorem}

\begin{remark}
In 2021, Cho, Hyun, O and Park \cite{Cho-Hyun-O-Park2021} conjectured that: Let $a\cdot n$ 
be an even integer at least $2$, where $n \geq a + 1$. If $G$ is a graph of order $n$ with 
$\lambda (G) > \lambda (H_{n,a})$, then $G$ contains an $[a, b]$-factor. Recently, this 
conjecture was confirmed by Fan, Lin and Lu \cite{Fan-Lin-Lu2022} for the case $n\geq 3a+b+1$. 
Finally, it was confirmed by Wei and Zhang \cite{Wei-Zhang2023} completely using different proof techniques.
\end{remark}

Lihua Feng (private communication) asked a tight spectral condition for a triangle factor
in a graph on $n=3k$ vertices. Setting $F = \frac{n}{r+1}K_{r+1}$ in 
Theorem \ref{thm:spectral-version-Alon-result}, we have

\begin{theorem}
Let $(r+1) \mid n$ and $F = \frac{n}{r+1}K_{r+1}$. Suppose that $n$ is sufficiently large.
If $G$ is an $n$-vertex graph not containing $F$ as a subgraph, then $\lambda (G)\leq\lambda (H_{n,r})$,
with equality holds if and only if $G = H_{n,r}$.
\end{theorem}

\subsection{A refined open problem related to Theorem \ref{Thm:WKX23}}

One may ask whether we can use an positive integer valued function $f(n)$
instead of the term ``$O(1)$" in Theorem \ref{Thm:WKX23} or not.
\begin{problem}
Let $r\geq 2$ be an integer, and $H$ be a graph with $ex(n,H)=e(T_{n,r})+f(n)$, 
where $f(n)=n^{\alpha}$ is an integer-value function, $\alpha>0$ is a real number.
Determine $\sup \alpha$, such that for sufficiently large $n$, we have $\SPEX (n,H)\subseteq \EX(n,H)$.
\end{problem}

\subsection{More counterexamples to Problem \ref{Prob:1}}

It seems that many bipartite graphs are counterexamples for Problem \ref{Prob:1}.
For example, a well-known result proved by F\"{u}redi \cite{Furedi1983} states that 
$\ex(q^2+q+1,C_4) = q(q+1)^2/2$ where $q=2^k$ and the unique graph is Erd\H{o}s-R\'enyi 
graph. On the other hand, Nikiforov \cite{Nikiforov2007} proved that 
$\SPEX (n, C_4) = \{K_1\vee (\frac{n-1}{2})K_2\}$ when $n$ is odd, and Zhai and 
Wang \cite{ZW12} proved that $\SPEX (n, C_4) = \{K_1\vee \big((\frac{n-2}{2})K_2\cup K_1\big)\}$ 
when $n$ is even. Obviously, $\SPEX (n, C_4)\nsubseteq \EX (n, C_4)$. There are also 
counterexamples for Problem \ref{Prob:1} when $F$ is a non-bipartite graph. It was 
shown in \cite{CDT22,Y21} that $\SPEX(n,W_{2k+1})\cap \EX(n, W_{2k+1})=\emptyset$ 
when $k=7$ or $k\geq 9$ and $n$ is sufficiently large. So, the general solution to 
Problem \ref{Prob:1} is very changeling and mysterious.

\subsection{The disjoint copies version of Problem \ref{Prob:1}}

It is also natural to consider disjoint copies version of Problem \ref{Prob:1}.

\begin{problem}\label{Prob:2}
Let $F$ be any graph, and $k$ be a fixed positive integer. Characterize all graphs $F$ such that
\begin{equation}\label{align:2}
\SPEX (n, kF)\subseteq \EX(n, kF)
\end{equation}
for sufficiently large $n$.
\end{problem}

For Problem \ref{Prob:2}, the solution is true when $F$ is a clique \cite{Ni-Wang-Kang2023}. 
It should be mentioned that $\EX (n,kK_{r+1})=T_{n-r+1,r}\vee K_{k-1}$ for sufficinelty large 
$n$, as shown by Simonovits \cite{Simonovits68} (for a short proof, see \cite[p.\,593]{Bollobas78}).

\subsection{A conjecture}
We highly suspect the following holds, which is motivated by \cite{LN23} and the current work.

\begin{conjecture}
Let $k$ be a fixed positive integer and $n$ be a sufficiently large integer. Let $F$ be a graph 
such that $\ex (n, F) = \frac{n^2}{2}-kn+O(1)$. Then we have $\SPEX (n, F)\subseteq \EX(n, F)$.
\end{conjecture}

Finally, we would like to mention that the $A_{\alpha}$-matrix of a graph $G$, as introduced by 
Nikiforov \cite{Nikiforov2017}, is defined as $A_{\alpha} (G):= \alpha D(G) + (1-\alpha) A(G)$, 
where $0\leq \alpha \leq 1$. It is easy to extend our results to the $A_{\alpha}$-matrix for 
$\alpha\in (0, 1/2)$. We refer the exercise to interested readers.

\end{document}